\documentclass[graybox]{svmult}

\usepackage{type1cm}        
\usepackage{makeidx}         
\usepackage{graphicx}        
\usepackage{multicol}        
\usepackage[bottom]{footmisc}

\usepackage{newtxtext}       %
\usepackage[varvw]{newtxmath}       

\usepackage[T1]{fontenc}
\usepackage{lscape}
\usepackage{url}
\usepackage{longtable}
\usepackage{graphicx}
\usepackage{latexsym}
\usepackage{colortbl}
\usepackage[all]{xy}

\usepackage[shortlabels]{enumitem}
\setlist{leftmargin=*}

\newcommand{\Aut}{\mbox{\rm Aut}}
\newcommand{\Inn}{\mbox{\rm Inn}}

\newcommand{\Ind}{\mbox{\rm Ind}}
\newcommand{\Infl}{\mbox{\rm Infl}}
\newcommand{\Res}{\mbox{\rm Res}}
\newcommand{\Irr}{\text{\rm Irr}}

\newcommand{\Lin}{\text{\rm Lin}}

\newcommand{\GL}{\text{\rm GL}}

\newcommand{\SL}{\text{\rm SL}}
\newcommand{\PGL}{\text{\rm PGL}}
\newcommand{\PGU}{\text{\rm PGU}}
\newcommand{\PSL}{\text{\rm PSL}}
\newcommand{\SU}{\text{\rm SU}}
\newcommand{\PSU}{\text{\rm PSU}}

\newcommand{\PSp}{\text{\rm PSp}}

\newcommand{\id}{\mbox{\rm id}}

\newcommand{\ad}{\mbox{\rm ad}}

\makeindex             

\begin{document}

\title*{The $R_{\infty}$-property of flat manifolds: Toward the eigenvalue one 
property of finite groups}
\titlerunning{The eigenvalue one property}
\author{Gerhard Hiss and\\ Rafa{\l} Lutowski}
\institute{Gerhard Hiss \at Lehrstuhl f{\"u}r Algebra und Zahlentheorie, 
RWTH Aachen University, 52056 Aachen, Germany, 
\email{gerhard.hiss@math.rwth-aachen.de}
\and Rafa{\l} Lutowski \at Institute of Mathematics, Faculty of
Mathematics, Physics and Informatics, University of Gda\'nsk, 
ul. Wita Stwosza 57,
80-308 Gda\'nsk, Poland, \email{rafal.lutowski@ug.edu.pl}}
%
%

\maketitle

\begin{dedication}
Dedicated to Otto Kegel
\end{dedication}

\abstract{We introduce a conjecture of Dekimpe, De Rock and Penninckx and sketch 
some major steps in its proof. This text presents an extended account of the 
talk of the first author given at the conference.}

\section{Introduction} The purpose of this note is to outline the proof, 
recently obtained by the authors in~\cite{HiLu}, of a conjecture of Dekimpe, 
De Rock and Penninckx. This yields a sufficient condition for a flat manifold to
be an $R_{\infty}$-manifold. Here, we will sketch the main steps of our proof, 
giving details for some of the more elementary arguments. We show how to reduce 
the conjecture to 
the finite simple groups. Then, using extensive, detailed knowledge about the 
automorphism groups, the subgroup structure and the character theory of the 
non-abelian finite simple groups, these can be ruled out as minimal 
counterexamples to the conjecture. For the necessary information we rely largely
on~\cite{GLS}. Our arguments work for the majority of these groups, but they may 
fail for particular small instances. These are treated with computational 
methods, using the systems Chevie~\cite{chevie, Michel} and GAP~\cite{GAP04}.

With respect to groups and characters we use standard notation. In particular,
$\Irr(G)$ denotes the set of irreducible complex characters of the finite 
group~$G$. Characters of $\mathbb{R}G$-modules are tacitly viewed as complex 
characters.

\section{Motivation}
\label{sec:1}
Let~$M$ be a real closed manifold with fundamental group~$\pi_1(M)$ and let
$f \colon M \rightarrow M$ be a homeomorphism. 
Several invariants can be attached to~$f$.
For example, the Reidemeister number~$R(f)$ of~$f$ is the number of 
$f_\#$-conjugacy classes on $\pi_1(M)$, where~$f_\#$ is the induced map on 
$\pi_1(M)$. By definition,~$R(f)$ is a positive integer or infinity. Other 
invariants are the Lefshetz number $L(f)$, defined as an alternating sum of 
traces of the maps induced by~$f$ on the homology groups of~$M$, or the Nielsen 
number $N(f)$, the number of fixed point classes of~$f$ on~$M$ under a certain 
equivalence relation. In particular,~$f$ has at least $N(f)$ fixed points.

The manifold $M$ is called an $R_\infty$-manifold, if $R(f) = \infty$ for 
every homeomorphism~$f$ of~$M$.

Let us consider the special case where $M$ is an \textit{infra-nilmanifold}, 
i.e., $M = \Gamma \backslash L$, where $L$ is a connected, simply connected, 
nilpotent Lie group, and $\Gamma \leq L \rtimes C$ is discrete, cocompact and 
torsion-free for some maximal compact subgroup $C \leq \Aut(L)$. 
Then $\pi_1(M) \cong \Gamma.$ In the case when $\Gamma \leq L$, i.e.~$M$ is a 
\textit{nilmanifold}, $R(f) = \infty$ implies that $L(f) = N(f) = 0$; see the 
introduction of~\cite{DDRP}. 

Let us even further specialize an infra-nilmanifold to the case when 
$L = \mathbb{R}^m$. Then~$M$ is a \textit{flat manifold}. Here, $C = O(m)$,
$\Gamma \cap \mathbb{R}^m \cong \mathbb{Z}^m$, and there is a finite group $G$ such 
that
$$1 \longrightarrow \mathbb{Z}^m \longrightarrow \Gamma \longrightarrow G 
\longrightarrow 1$$
is a short exact sequence. Conjugation of~$\Gamma$ on~$\mathbb{Z}^m$ induces a 
homomorphism $\rho \colon G \rightarrow \GL_m( \mathbb{Z} )$, the corresponding 
\textit{holonomy representation.}

Let $M$ be a flat manifold, and $\rho$ the corresponding holonomy representation.
A $\mathbb{Z}$-subrepresentation of~$\rho$ is a
representation $\rho' \colon G \rightarrow \GL_d( \mathbb{Z} )$ arising from a
$\rho(G)$-invariant, pure sublattice $Y \leq \mathbb{Z}^m$ of rank~$d$;
here,~$Y$ is called pure, if some $\mathbb{Z}$-basis of~$Y$ extends to a
$\mathbb{Z}$-basis of~$\mathbb{Z}^m$. We may optionally view~$\rho$ or $\rho'$ 
as a $\mathbb{Q}$-representation or an $\mathbb{R}$-representation of~$G$ in a
natural way.

\begin{theorem}[Dekimpe, De Rock, Penninckx, 2009, \cite{DDRP}]
\label{thm:1}
Suppose there is a $\mathbb{Z}$-sub\-re\-pre\-sent\-ati\-on 
$\rho' \colon G \rightarrow \GL_d( \mathbb{Z} )$, which is an irreducible component
of~$\rho$ of multiplicity one as a $\mathbb{Q}$-subrepresentation, and such that
the following two conditions are satisfied:

\begin{enumerate}[{\rm (i)}, widest=ii]
\item If~$\rho''$ is a $\mathbb{Q}$-subrepresentation of~$\rho$ of
degree~$d$ such that $\rho'(G)$ and $\rho''(G)$ are conjugate in
$\GL_d( \mathbb{Q} )$, then~$\rho'$ and~$\rho''$ are equivalent.

\item For all $D \in N_{\GL_d( \mathbb{Z} )}( \rho'(G) )$ there is $g \in G$
such that $\rho'(g)D$ has eigenvalue $1$.
\end{enumerate}
Then~$M$ is an $R_\infty$ manifold. \hfill{$\Box$}
\end{theorem}

Condition~(ii) on $\rho'$ in Theorem~\ref{thm:1} implies:

\begin{itemize}
\item The $\mathbb{Q}$-representation $\rho'$ is $\mathbb{R}$-irreducible.

\item The normalizer $N_{\GL_d( \mathbb{Z} )}( \rho'(G) )$ has finite order.
\end{itemize}
These observations follow, e.g., from the proof of \cite[Theorem~A]{SzczOut}.

\section{The eigenvalue one condition}
\label{sec:2}
Motivated by Theorem~\ref{thm:1}, its authors formulated a conjecture which we 
are now going to introduce. Before doing so, we fix some notation which will be
kept throughout this article.

Let $G$ be a finite group, $V$ a finite-dimensional $\mathbb{R}G$-module and
$\rho \colon G \rightarrow \GL(V)$ the representation afforded by~$V$. Moreover, 
$n \in \GL(V)$ is an element of finite order normalizing $\rho(G)$.

\begin{definition}
Let the notation be as introduced above. We then say that:
\begin{enumerate}[1), widest*=3]
\item The triple $(G,V,n)$ has the \textit{$E1$-property}, if there is $g \in G$ 
such that $\rho(g)n$ has eigenvalue~$1$.

\item The pair $(G,V)$ has the \textit{$E1$-property}, if $(G,V,n')$ has the 
$E1$-property for all $n' \in \GL(V)$ of finite order normalizing $\rho(G)$.

\item The group $G$ has the \textit{$E1$-property}, if $(G,V')$ has the 
$E1$-property for all \textbf{irreducible}, \textbf{non-trivial}
$\mathbb{R}G$-modules $V'$ of \textbf{odd dimension}.
\end{enumerate}
\end{definition}

Let us give some examples.

\begin{example}
If $V = \mathbb{R}$ with trivial action of~$G$, i.e.~$V$ is the trivial
$\mathbb{R}G$-module, then $V$ does not have the $E1$-property. Indeed, 
$\rho(G) = \{ 1 \}$ in this case, and $n = -1$ violates the eigenvalue 
one condition.
\end{example}

\begin{example}
\label{exa:2}
If~$G$ is an elementary abelian $p$-group, then~$G$ has the $E1$-property.
Indeed, if~$p$ is odd, then~$G$ does not have any non-trivial, irreducible,
odd-dimensional module over~$\mathbb{R}$. If $p = 2$ and~$V$ is non-trivial 
and irreducible, then $\dim(V) = 1$ and
$\rho( G ) = \{ \pm 1 \}$, which proves our claim, as the only elements of
finite order in $\mathbb{R}^* = \GL( V )$ are $\pm1$.
\end{example}

\begin{example}[Dekimpe, De Rock, Penninckx, 2009,  \cite{DDRP}]
Let $G$ be the extraspecial $2$ group $2_+^{1+4}$ of 
order~$32$, and let~$V$ be the irreducible $\mathbb{R}G$-module of 
dimension~$4$. Then $(G,V)$ does not have the $E1$-property.
\end{example}

With these examples in mind, we can now formulate the following conjecture, 
which is a slight generalization of the original conjecture in 
\cite[Conjecture~$4.8$]{DDRP}.

\begin{conjecture}[Dekimpe, De Rock, Penninckx, 2009, \cite{DDRP}]
\label{conj:1}
Every finite group has the $E1$-property. \hfill{$\Box$}
\end{conjecture}

\section{The main theorem}

The purpose of this article is to announce the proof of Conjecture~\ref{conj:1}.

\begin{theorem}[\cite{HiLu}]
\label{thm:2}
Every finite group has the $E1$-property. \hfill{$\Box$}
\end{theorem}
In view of Theorem~\ref{thm:1}, this has the following consequence.
\begin{corollary}
\label{cor:1}
Let~$M$ be a flat manifold with holonomy representation 
$\rho \colon G \rightarrow \GL_n( \mathbb{Z} )$. Suppose there is a non-trivial
$\mathbb{Z}$-subrepresentation $\rho' \colon G \rightarrow \GL_d( \mathbb{Z} )$ 
of odd degree~$d$, which is irreducible and of multiplicity one as an
$\mathbb{R}$-subrepresentation of~$\rho$.  
Suppose further that~$\rho'$ satisfies Condition~{\rm (i)} 
of\,\,{\rm Theorem~\ref{thm:1}}.

Then~$M$ is an $R_{\infty}$-manifold. \hfill{$\Box$}
\end{corollary}

Corollary~\ref{cor:1} for solvable groups~$G$ has been proved by Lutowski and
Szczepa{\'n}ski in \cite[Theorem~$1.4$]{LuSzcz}.

The proof of Theorem~\ref{thm:2} uses the classification of the finite simple 
groups. We are now going to sketch the main steps, mostly without proofs.

\subsection{The restriction method}
\label{RestrictionMethod}
Let us begin to set up our notation, which will be valid throughout this 
subsection. Let~$G$ be a finite group and~$V$ a non-trivial, odd-dimensional
$\mathbb{R}G$-module. Write $\rho \colon G \rightarrow \GL(V)$ for the representation 
afforded by $V$. The assertion of the conjecture only concerns the 
image~$\rho(G)$, so we may and will, for the
rest of this subsection, assume that $\rho$ is faithful. We identify~$G$ with
$\rho(G) \leq \GL(V)$. Also, $n \in N_{\GL(V)}(G)$ is an element of finite
order.

\begin{remark}
\label{rem:0}
If~$V$ is irreducible, it is absolutely irreducible. \hfill{$\Box$}
\end{remark}

It is natural to search for elements $g \in G$ such that~$gn$ has eigenvalue~$1$
in suitable subgroups of~$G$, to which we can apply an inductive hypothesis.

\begin{lemma}
\label{lem:1}
Let $H \leq G$. Suppose that the following conditions are satisfied.
\begin{itemize}
\item The group $H$ is $n$-invariant.
\item There is $V_1 \leq V$, $H$-invariant and $n$-invariant.
\item The triple $(H,V_1,n)$ has the $E_1$-property.
\end{itemize}
Then $(G,V,n)$ has the $E1$-property.
\end{lemma}
\begin{proof}
Choosing a basis of~$V$ through~$V_1$, the elements of~$H$ and $n$ are represented
by matrices of the following shape:
$$H = \left\{ \left( \begin{array}{c|c} * & * \\ \hline 0 & * \end{array} \right) \right\},\quad\quad
n = \left( \begin{array}{c|c} n_1 & * \\ \hline 0 & * \end{array} \right).$$
Since $(H,V_1,n_1)$ has the $E1$-property, there is
$$h = \left( \begin{array}{c|c} h_1 & * \\ \hline 0 & * \end{array} \right) \in H,$$
such that $h_1n_1$ has eigenvalue~$1$.
Thus $hn$ has eigenvalue $1$.
\end{proof}

In the following two lemmas, we present two important applications of the 
restriction method.

\begin{lemma}
\label{lem:2}
Let $S \leq V$ be an irreducible $\mathbb{R}G$-submodule such that the following
conditions hold.
\begin{itemize}
\item The module $V$ is $S$-homogeneous.
\item The pair $(G,S)$ has the $E1$-property.
\end{itemize}
Then $(G,V)$ has the $E1$-property.
\end{lemma}
\begin{proof} Recall that $n \in \GL(V)$ is an arbitrary element of finite order 
normalizing~$G$. Put $A := \langle G, n \rangle \leq \GL(V)$. Let $V_1' \leq V$ be 
an irreducible $\mathbb{R}A$-submodule of~$V$ of odd dimension and let 
$V_1 \leq V_1'$ be an irreducible $\mathbb{R}G$-submodule of~$V_1'$. Then~$V_1$
and~$V_1'$ are absolutely irreducible by Remark~\ref{rem:0}.

By hypothesis, $V_1 \cong S \cong nV_1$ as $\mathbb{R}G$-modules. As $A/G$ is 
cyclic and~$V_1$ is 
absolutely irreducible, the character of~$V_1$ extends to~$A$. Then, by Clifford 
theory, every absolutely irreducible $\mathbb{R}A$-submodule of 
$\Ind_G^A( V_1 )$ has dimension $\dim( V_1 )$. Hence $V_1 = V_1'$, and 
thus~$V_1$ is $n$-invariant. The claim follows from Lemma~\ref{lem:1}, applied 
with $H = G$.
\end{proof}

\begin{lemma}
\label{lem:3}
Let $H \unlhd G$ be characteristic in $G$.
Suppose that~$V$ is irreducible and 
let $S$ be an irreducible $\mathbb{R}H$-submodule of $V$.

If $(H,S)$ has the $E1$-property, so does $(G,V)$.
\end{lemma}
\begin{proof} The group~$G$ permutes the homogeneous components of 
$\Res^G_H( V )$ transitively. Hence the $S$-homogeneous component~$V_1$ of 
$\Res^G_H( V )$ has odd dimension, and there exists $g \in G$ such that 
$gnV_1 = V_1$.

Thus $gn \in \GL(V_1)$ has finite order and normalizes~$H$. 
By~Lemma~\ref{lem:2}, the triple $(H,V_1,gn)$ has the $E1$-property. In turn,
$(G,V,gn)$ has the $E1$-property by Lemma~\ref{lem:1}. As~$n$ was arbitrary, 
this proves our assertion.
\end{proof}

\subsection{Reduction to the finite simple groups}
\label{ReductionToSimpleGroups}
We indicate how the following proposition, whose proof we omit, and the 
restriction method yield a reduction of the main theorem to the finite simple 
groups. 

\begin{proposition}
\label{prop:1}
Let $G = L \times \cdots \times L$, where~$L$ is a non-abelian finite
simple group, i.e.\ $G$ is non-abelian and characteristically simple.

If~$L$ has the $E1$-property, then $G$ has. \hfill{$\Box$}
\end{proposition}

\begin{corollary}
\label{cor:2}
A minimal counterexample to {\rm Conjecture~\ref{conj:1}} is a non-abelian 
finite simple group.
\end{corollary}
\begin{proof} 
Let~$G$ be a minimal counterexample. If $H \lneq G$ is characteristic,~$H$ has 
the $E1$-property by hypothesis. Then~$G$ has the $E1$-property by 
Lemma~\ref{lem:3}.

Thus~$G$ is characteristically simple. But~$G$ is non-abelian by 
Example~\ref{exa:2}. Hence~$G$ is simple by Proposition~\ref{prop:1}.
\end{proof}

\begin{corollary}
\label{cor:3}
A solvable group has the $E1$-property.
\end{corollary}

\subsection{On the structure of the problem}
\label{FirstObservations}
Here, we will present some elementary observations which help to identify the
structure of the problem, and which prepare for a further method to be
introduced below. 

We return to the hypotheses of Subsection~\ref{RestrictionMethod}. Thus
$G \leq \GL(V)$ is a finite group, where~$V$ is a non-trivial 
$\mathbb{R}G$-module of odd dimension. Moreover, $n \in N_{\GL(V)}(G)$ is an 
element of finite order. In addition, we assume that $V$ is irreducible. 
Recall from Remark~\ref{rem:0} that~$V$ is absolutely irreducible.

With these notations, we note the following easy facts.

\begin{lemma}
\label{lem:4}
The following statements hold.
\begin{itemize}
\item We have $C_{\GL(V)}(G) = \{ x \cdot \id_V \mid x \in \mathbb{R} \}$, and 
$N_{\SL(V)}(G)$ embeds into $\Aut(G)$.

\item The set of elements of finite order in $N_{\GL(V)}(G)$
equals $N_{\SL(V)}(G) \times \langle -\id_V \rangle$.

\item If $\det(n) = 1$, then $n$ has eigenvalue $1$. \hfill{$\Box$}
\end{itemize}
\end{lemma}
The first item of this lemma implies that 
$Z( G ) \leq \langle \pm \id_V \rangle$. 
With respect to 
the last item of this lemma, we emphasize that $-n$ is an element of finite 
order normalizing~$G$, and $\det(-n) = -1$, as $\dim(V)$ is odd.
In particular, if $-\id_V \in G$, then $(G,V,n)$ has the $E1$-property.
So let us assume that $-\id_V \not\in G$ in the following.
Then~$Z(G)$ is trivial.

Let us introduce further notation. For $g \in N_{\GL(V)}( G )$, we write 
$\ad_g \in \Aut(G)$ for the automorphism induced by conjugation with~$g$, 
and put $\nu := \ad_n$. Notice that $\nu = \ad_{-n}$. As in the proof of 
Lemma~\ref{lem:2}, set $A := \langle G, n \rangle \leq \GL(V)$. Finally, let 
$A_1 := A \cap \SL(V)$. We distinguish two cases:

\begin{description}
\item[\textit{Case~$1$}:] We have $-\id_V \not\in A$.

\item[\textit{Case~$2$}:] We have $-\id_V \in A$. 
\end{description}

\noindent
Notice that $A = A_1 \times \langle -\id_V \rangle$ in Case~$2$. It is helpful 
to take a more abstract point of view.
Put 
$$A' := \begin{cases} A, & \text{in Case\ } 1, \\ A_1, & \text{in Case\ } 2, 
\end{cases}$$
and 
$$G' := \langle \Inn(G), \nu \rangle \leq \Aut(G).$$

\begin{remark}
\label{rem:1}
There is a surjective homomorphism 
$$\rho' \colon A \rightarrow G'$$
with
$$
gn^i \mapsto \ad_g \circ \nu^i \text{\ for\ } g \in G \text{\ and\ } 
i \in \mathbb{Z}. 
$$
Moreover,~$\rho'$ restricts to an isomorphism $A' \rightarrow G'$. \hfill{$\Box$}
\end{remark}

Let $\chi \in \Irr(G)$ and $\chi' \in \Irr(A)$ denote the irreducible characters 
of~$G$, respectively~$A$, afforded by~$V$. We also write~$\chi'$ for the 
restriction of~$\chi'$ to~$A'$. 
The isomorphism $(\rho'|_{A'})^{-1} \colon G' \rightarrow A'$ from Remark~\ref{rem:1}
makes~$V$ into an $\mathbb{R}G'$-module, and, by a slight abuse of notation, we
also let~$\chi'$ denote the character of~$G'$ afforded by~$V$. Thus 
$\chi'( \rho'( a' ) ) = \chi'( a' )$ for all $a' \in A'$.

\subsection{The large degree method} 
\label{LargeDegreeMethod}
Keep the hypotheses and notation of Subsection~\ref{FirstObservations}. 
Notice that a cyclic group of even order contains exactly two real absolutely
irreducible characters.
  
\begin{proposition}
\label{prop:2}
Suppose there is $g \in G$ such that $\alpha := \ad_g \circ \nu$ has even
order and $(\Res^{G'}_{\langle \alpha \rangle}( \chi' ),\lambda) > 0$ for
every real $\lambda \in \Irr( \langle \alpha \rangle )$.
Then $(G,V,n)$ has the $E1$-property.
\end{proposition}
\begin{proof}
Let~$\rho'$ denote the homomorphism from Remark~\ref{rem:1}. 

Suppose that we are in Case~$1$. Then $\rho'( gn ) = \alpha$.  By hypothesis, 
$\Res^{A'}_{\langle gn \rangle}( \chi' )$ contains the trivial character 
of~$\langle gn \rangle$, and thus~$gn$ has eigenvalue~$1$.

Suppose that we are in Case~$2$. Let $n_1 \in \{ \pm n \}$ such that 
$gn_1 \in A'$. Then $\rho'( gn_1 ) = \alpha$. By hypothesis, 
$\Res^{A'}_{\langle gn_1 \rangle}( \chi' )$ contains each of the two real 
irreducible characters of~$\langle gn_1 \rangle$, and thus~$gn_1$ has 
eigenvalue~$1$ and~$-1$. Hence $gn \in \{ gn_1, -gn_1 \}$ has eigenvalue~$1$.
\end{proof}

\begin{corollary}
\label{cor:4}
Suppose there is $g \in G$ such that $\ad_g \circ \nu$ has order~$2$. Then 
$(G,V,n)$ has the $E1$-property. \hfill{$\Box$}
\end{corollary}

These observations lead to the \textit{Large Degree Method}, which is 
formulated in the following proposition.

\begin{proposition}
\label{prop:3}
Suppose there is $g \in G$ such that $\alpha := \ad_g \circ \nu$ has even
order and
$$\dim(V) > (|\alpha| - 1)|C_G( \alpha' )|^{1/2}$$
for every $\alpha' \in \langle \alpha \rangle$ of prime order.
Then $(G,V,n)$ has the $E1$-property.
\end{proposition}
\begin{proof}
The second orthogonality relation implies
$|\chi'( \alpha' )| \leq |C_G( \alpha' )|^{1/2}$ for every nontrivial
$\alpha' \in \langle \alpha \rangle$.
Our hypothesis implies 
$(\Res^{G'}_{\langle \alpha \rangle}( \chi' ),\lambda) > 0$ for every
$\lambda \in \Irr( \langle \alpha \rangle )$.
Now use Proposition~\ref{prop:2}.
\end{proof}

\begin{corollary}
\label{cor:5}
Let $G$ b one of the following simple groups:
\begin{itemize}
\item the Tits group;
\item a sporadic simple group;
\item an alternating group $A_n$ with $n \geq 5$ and $n \neq 6$.
\end{itemize}
Then $G$ has the $E1$-property.
\end{corollary}
\begin{proof} 
In these cases, $\Aut(G) = \Inn(G) \rtimes \Phi$ with $|\Phi| \leq 2$. The claim 
follows from Corollary~\ref{cor:4}.
\end{proof}
The group $A_6$ omitted here  will be treated as the Chevalley group $\PSL_2(9)$.

\begin{example}
\label{exa:4}
Let $G = G_2(q)$, the simple Chevalley group type~$G_2$ with $q = 3^f$, and 
let~$V$ be the the Steinberg module of~$G$ over~$\mathbb{R}$. 

Then $|G| = q^6(q^2-1)(q^6-1)$ and $\dim(V) = q^6$. It is known that 
$\Aut(G) = \Inn(G) \rtimes \Phi$, where $\Phi$ is cyclic of order $2f$.
(This is the reason for taking $q$ to be a $3$-power, as otherwise
$\Phi$ is cyclic of order $f$ and the arguments become easier.)
There is $h \in G$ such that $\ad_h \circ \nu = \mu \in \Phi$.

If $|\mu|$ is even, put $\alpha := \mu$. Otherwise, let $u \in G$ be a
$\mu$-stable involution, and put $\alpha := \ad_u \circ \mu = \ad_{uh} 
\circ \nu$. Such a $\mu$-stable involution exists, as the set of $\mu$-fixed
points in~$G$ is a Chevalley group of type $G_2$ or a twisted group
${^3G}_2( 3^{2m+1} )$ for some non-negative integer~$m$;
see \cite[Proposition~$4.9.1$(a)]{GLS}.

Then $|\alpha|$ is even, $|\alpha| \leq 2f$, and $|C_G( \alpha' )| \leq q^7$
for every $\alpha' \in \langle \alpha \rangle$ of prime order. This follows
from the known fixed point groups of elements of $\Aut(G)$; see
\cite[Propositions $4.9.1$, $4.9.2$]{GLS}. Since
$$\dim(V) = q^6 > (|\alpha| - 1)q^{7/2},$$
$(G,V)$ has the $E1$-property.
\end{example}

\subsection{The simple groups of Lie type}
\label{SimpleGroupsOfLieType}
By Corollary~\ref{cor:5} and the classification of the finite simple groups, we 
have to rule out the simple groups of Lie type as minimal counterexamples to 
Conjecture~\ref{conj:1}. Here is a list of these groups, omitting the conditions
on simplicity:

\begin{itemize}
\item \textit{classical groups}: $\PSL_{d}(q)$, $\PSU_{d}(q)$, $P\Omega_{2d+1}(q)$,
$\PSp_{2d}(q)$, $P\Omega^+_{2d}(q)$, $P\Omega^-_{2d}(q)$;

\item \textit{exceptional Chevalley groups}: $G_2(q)$, $F_4(q)$, $E_6(q)$, $E_7(q)$, $E_8(q)$;

\item \textit{twisted groups}: ${^3\!D}_4(q)$, ${^2\!E}_6(q)$.
\end{itemize}

In all these cases,~$q$ is a power of a prime $r$, the \textit{characteristic} 
of~$G$. We further have the following series of groups:

\begin{itemize}

\item \textit{Ree groups}: ${^2G}_2(3^{2m+1})$, ${^2\!F}_4(2^{2m+1})$;

\item \textit{Suzuki groups}: ${^2\!B}_2(2^{2m+1})$.
\end{itemize}

\subsection{Groups of Lie type of odd characteristic}
\label{OddCharacteristic}
Let~$G$ be a finite group of Lie type, of odd characteristic, i.e.~$r$ is odd in 
the notation of Subsection~\ref{SimpleGroupsOfLieType}, or~$G$ is a Ree group 
${^2G}_2(3^{2m+1})$. Let $r = 3$ in the latter case. Then~$G$ is a group with 
a split $BN$-pair of characteristic~$r$; see, e.g.,~\cite[Subsection $2.5$]{C2}.

In particular, there are distinguished subgroups $B$, $U$ and $T$, with
$B = U \rtimes T$, where $U = O_r( B )$. The groups~$B$ and~$T$ are the
\textit{standard Borel subgroup}, respectively the \textit{standard maximal 
torus} of~$G$.
In classical groups,~$B$ arises from the group of upper triangular matrices,
and~$T$ from the group of diagonal matrices. More generally, a group~$P$ with
$B \leq P$ is a \textit{standard parabolic subgroup}. This has a \textit{Levi 
decomposition} $P = O_r( P ) \rtimes L$, where~$L$ is a \textit{standard Levi 
subgroup} of~$G$, and~$O_r( P )$ is the \textit{unipotent radical} of~$P$. The 
following lemma is the main tool to deal with the groups of Lie type of odd
characteristic. It will be applied to the character $\chi := \chi_V$ of the 
$\mathbb{R}G$-module~$V$ under consideration; see the introduction to 
Subsection~\ref{RestrictionMethod}. Notice that~$\chi$ is real valued, 
non-trivial and of odd degree.

\begin{lemma}
\label{lem:5}
Let $P = O_r( P ) \rtimes L$ be a standard parabolic subgroup of~$G$ with 
standard Levi subgroup~$L$. Let $\chi \in \Irr( G )$ with~$\chi$ real 
and~$\chi(1)$ odd. Then there exists $\lambda \in \Irr(P)$ such that
the following conditions hold.

\begin{itemize}
\item The unipotent radical $O_r( P )$ is in the kernel of~$\lambda$ 
(i.e.\ $\lambda \in \Irr(L)$).
\item The scalar product $(\lambda, \Res^G_P( \chi ) )$ is odd.
\item If $P = B$, i.e.\ $L = T$, then $\lambda^2 = 1_T$ (in the character group 
of $T$).
\end{itemize}
\end{lemma}
\begin{proof}
Let $V$ be an $\mathbb{R}G$-module with character $\chi$. Let 
$V_1 \leq \Res^G_P( V )$ be a homogeneous component of odd dimension. Let 
$S \leq V$ be a simple $\mathbb{R}P$-submodule of~$V_1$. Then $\dim(S)$ is odd, 
hence $S$ is absolutely irreducible by Remark~\ref{rem:0}. Let~$\lambda$ be the 
character of~$S$. Then the scalar product $(\lambda, \Res^G_P( \chi ) )$ divides 
the dimension of~$V_1$, hence is odd.

Now consider the restriction of~$S$ to~$O_r(P)$. Let~$S'$ denote a simple 
$\mathbb{R}O_r(P)$-submodule of $\Res^P_{O_r(P)}( S )$. Then $\dim(S')$ is 
odd,~$S'$ is absolutely irreducible, and its character is real. Hence~$S'$ is 
the trivial module and $O_r(P)$ is in the kernel of~$\lambda$.

If $L = T$, which is abelian, then $\lambda(1) = 1$ and $\lambda^2 = 1_T$.
\end{proof}
The proof uses the fact that~$r$ is odd, i.e.\ that~$|O_r(P)|$ is odd, in an 
essential way. Namely, the only irreducible $\mathbb{R}H$-module of odd
dimension of a group~$H$ of odd order is the trivial module.

We now sketch how Lemma~\ref{lem:5} is applied. Assume first that $P = B$, hence 
$L = T$ and $O_r(P) = U$, and let $\chi$ and $\lambda$ be as in this lemma. Then 
$(\lambda, \Res^G_B( \chi ) )$ is odd, in particular non-zero. By Frobenius 
reciprocity, $(\Ind_B^G( \lambda ), \chi ) > 0$. Moreover,~$U$ is in the kernel 
of~$\lambda$, so that~$\lambda$ may be thought of as the inflation 
$\Infl_T^B( \lambda )$, where~$\lambda$ now stands for the restriction 
of~$\lambda$ to~$T$. Thus~$\chi$ is a constituent of 
$$R_T^G( \lambda ) := \Ind_B^G( \Infl_T^B( \lambda ) ).$$
The map $R_T^G$ defined in the above equation is called \textit{Harish-Chandra
induction}. The constituents of $R_T^G( \lambda )$ for $\lambda  \in \Irr(T)$ 
are classified by Harish-Chandra theory; see~\cite[Section~$9$,~$10$]{C2}. 
Those arising for $\lambda \in \Irr( T )$ with $\lambda^2 = 1_T$ are rare. For
example, if $G = E_6 (q)$ with $q$ odd, there are~$8$ irreducible, real 
characters of odd degree, whereas $|\Irr(G)| = q^6 + q^5 + \text{\rm\ (lower 
terms in~$q$)}$.
The constituents of $R_T^G( 1_T )$, i.e.\ for  $\lambda = 1_T$, are called 
\textit{principal series characters}. These are, in particular, \textit{unipotent
characters} of~$G$. Unipotent characters of odd degree are known in each case. 
The Steinberg character is one of these. Thus a first application of 
Lemma~\ref{lem:5} reduces the number of $\mathbb{R}G$-modules~$V$ to be 
investigated drastically.

We aim to apply the restriction method of Lemma~\ref{lem:1} with $H = P$, a 
standard parabolic subgroup of~$G$. If possible, we choose~$P$ in such a way 
that for every $\alpha \in \Aut(G)$, the pair $(\alpha(P),\alpha(L))$, 
where~$L$ is the standard Levi subgroup of~$P$, is conjugate to~$(P,L)$ in~$G$. 
This condition, which is always satisfied for $P = B$ and $L = T$, can be 
verified by using the description of $\Aut(G)$ 
(see \cite[Theorem~$2.5.12$]{GLS}) and the construction of the standard 
$BN$-pair of~$G$ (see \cite[Subsections~$1.11$,~$2.3$]{GLS}). 

To continue, assume that $G \leq \GL(V)$, where~$V$ is a non-trivial irreducible 
$\mathbb{R}G$-module of odd dimension. Let $n \in \GL(V)$ be of finite order 
normalizing~$G$. Let $\chi = \chi_V$ be 
the character of~$G$ afforded by~$V$. Suppose that~$P$ and~$L$ are chosen as 
above. Then, by replacing~$n$ with~$gn$ for a suitable $g \in G$, we may assume 
that~$n$ fixes~$P$ and~$L$.

Assume first that~$\chi$ is not a principal series character. In this case, 
choose $P = B$ and $L = T$, and let~$\lambda$ be as in Lemma~\ref{lem:5}. Let 
$V_1 \leq V$ denote the $\lambda$-homogeneous component of $\Res^G_B( V )$. 
Then $\lambda \neq 1_T$ and $\dim( V_1 )$ is odd. Hence~$V_1$ has the 
$E1$-property by Lemma~\ref{lem:2} and Corollary~\ref{cor:3}, as~$B$ is 
solvable. However,~$V_1$ is not, a priori, invariant 
under~$n$. In this case, we can usually replace~$V_1$ by an $n$-invariant 
$N$-conjugate~$V_1'$ in~$V$, where $N \leq N_G(T)$ is the group~$N$ from the 
$(B,N)$-pair of~$G$. Then the restriction method applies with $H = B$ 
and~$V_1'$. The few instances, where there is no $n$-invariant $N$-conjugate 
of~$V_1$ only occur for $G = \PSL_d( q )$, and are treated by replacing~$B$ by 
a slightly larger parabolic subgroup.

Assume now that~$\chi$ is a principal series character. Here, we cannot use the
approach from the previous paragraph, as then~$V_1$, being a direct sum of
trivial modules, does not have the $E1$-property. Instead, we choose a parabolic
subgroup~$P$ as above, such that $\Ind_P^G( \mathbb{R} )$ contains~$V$ with even 
multiplicity (including~$0$), where~$\mathbb{R}$ denotes the trivial 
$\mathbb{R}P$-module. Then $\Res_P^G( V )$ contains a homogeneous component~$V_1$ 
of odd dimension, which is not a direct sum of trivial modules. A simple 
$\mathbb{R}P$-submodule $S \leq V$ has $O_r(P)$ in its kernel by 
Lemma~\ref{lem:5}, and, viewed as an $\mathbb{R}L$-module,~$S$ is in the
principal series of~$L$. The $n$-invariance of~$V_1$ is satisfied, since, for 
the chosen groups~$P$, the principal series $\mathbb{R}L$-modules are invariant 
under automorphisms of~$L$; see \cite[Theorem~$2.5$]{MalleExt}.

This approach fails for groups of small rank and~$V$ the Steinberg module. The 
worst case is $G = \PSL_2( q )$, which has to be treated in an ad hoc manner. In 
other cases, the large degree method as in Example~\ref{exa:4} can be applied.

This way we can rule out the simple groups of Lie type of odd characteristic as 
minimal counterexamples to Conjecture~\ref{conj:1}.

\subsection{Groups of Lie type of even characteristic}
\label{EvenCharacteristic}

Now let~$G$ be a finite simple group of Lie type of even characteristic, 
i.e.~$q$ is even in the notation of Subsection~\ref{SimpleGroupsOfLieType} 
or~$G$ is a Ree group ${^2\!F}_4(2^{2m+1})$ or a Suzuki 
group~${^2\!B}_2(2^{2m+1})$. In this case, most of the irreducible characters 
of~$G$ have odd degree. For example, if $G = E_6(q)$, with $q$ even, there are
$q^6 + 8q^2$ of them.

One of the issues in this case is to parametrize the odd degree irreducible 
characters and to find the reals among them. This is achieved with Lusztig's
generalized Jordan decomposition of characters. Before we discuss this in more 
detail, we consider a special situation which simplifies the problem.

As always, assume that $G \leq \GL(V)$ for some non-trivial irreducible 
$\mathbb{R}G$-module~$V$ of odd dimension. Let $n \in \GL(V)$ be of finite order 
normalizing~$G$. Let~$\nu$ denote the automorphism of~$G$ induced by conjugation 
with~$n$. Let $\chi = \chi_V$ be the character of~$G$ afforded by~$V$. Again, 
let~$B$ denote the standard Borel subgroup of~$G$ and~$T$ its standard maximal 
torus. Then $B = U \rtimes T$ with $U = O_2( B )$. As in 
Subsection~\ref{OddCharacteristic}, we may assume that~$\nu$ fixes $U$ and~$T$.

Put $\Lin(U) := \Irr( U/[U,U] )$, the set of linear characters of~$U$.

\begin{proposition}
\label{prop:4}
Suppose that $\nu^2$ fixes every $T$-orbit on $\Lin(U)$.
Then $(G,V,n)$ has the $E1$-property. 
\end{proposition}
\begin{proof}
Since $\chi(1)$ is odd, there is a $T$-orbit $\mathcal{O}$ on $\Lin(U)$
such that $|\mathcal{O}|(\lambda,\Res^G_U( \chi ))$ is odd for every
$\lambda \in \mathcal{O}$. Moreover, the number of such orbits is odd.

By hypothesis, there is a $\nu$-stable such orbit $\mathcal{O}$. Let 
$\lambda \in \mathcal{O}$. Suppose that $\lambda = 1_U$. Then the set~$V^U$ 
of $U$-fixed point on~$V$ has odd dimension $(\lambda,\Res^G_U( \chi ))$. 
As~$V^U$ is an $\mathbb{R}B$-module, some irreducible 
$\mathbb{R}B$-constituent~$S$ of $\Res_B^G( V )$ with~$U$ in its kernel has 
odd dimension. As~$|T|$ is odd,~$S$ must be the trivial module. Thus 
$\Res_B^G( \chi )$ contains a trivial constituent. In turn,~$\chi$ is a 
principal series character. However, the only principal series character of
odd degree is the trivial character by \cite[Theorem~$6.8$]{MalleHeight0}. But 
$\chi \neq 1_G$ by hypothesis. Thus $\lambda \neq 1_U$.

Since~$U/[U,U]$ is an elementary abelian $2$-group,~$\lambda$ is the character 
of an irreducible $\mathbb{R}U$-module. Let~$V_1$ denote the 
$\lambda$-homogeneous component of~$\Res_U^G( V )$. Then~$V_1$ has odd dimension 
$(\lambda,\Res^G_U( \chi ))$. Let $S$ be a simple $\mathbb{R}U$-submodule 
of~$V_1$. The character of~$nS$ is the $\nu$-conjugate of~$\lambda$. Since the 
$T$-orbit of~$\lambda$ is $\nu$-invariant, there is $t \in T$ such that 
$tnS \cong S$ as $\mathbb{R}U$-modules.
It follows that $tnV_1 = V_1$. As~$U$ is solvable, it has the $E1$-property by 
Corollary~\ref{cor:3}. Lemma~\ref{lem:1} applied with $H = U$ and $n$ replaced 
by~$tn$ yields our assertion.
\end{proof}

\subsection{The remaining groups}
Using Proposition~\ref{prop:4}, one can rule out the simple groups of Lie type 
of even characteristic as minimal counterexamples to Conjecture~\ref{conj:1}, 
except those in the following list (where~$q$ is even in each case):

\begin{enumerate}[(i), widest=iii]
\item $G = \PSL_d( q )$ with $d \geq 3$ and $\gcd( d, q - 1) > 1$;
\item $G = \PSU_d( q )$ with $d \geq 3$ and $\gcd( d, q + 1) > 1$;
\item $G = E_6(q)$ with $3 \mid q - 1$;
\item $G = {^2\!E}_6(q)$ with $3 \mid q + 1$;
\item $G = P\Omega_8^+(q)$.
\end{enumerate}
For these groups, there exist non-trivial, non-involutary, inner-diagonal, 
respectively graph automorphisms, that do not satisfy the hypothesis of 
Proposition~\ref{prop:4} in general. To rule out these remaining groups, is 
more than half of the entire work. Let us give a very rough sketch of the main 
ideas, concentrating on the groups~$G$ in the first four of the cases. The group 
$P\Omega_8^+(q)$ is treated in a similar way.

Choose a $\sigma$-setup for~$G$ according to
\cite[Definition~$2.2.1$]{GLS}. That is, let $\overline{G}$ be a simple 
algebraic group, of adjoint type, over the algebraic closure~$\mathbb{F}$ of the 
field with~$2$ elements. Let~$\sigma$ be a Steinberg morphism of~$\overline{G}$ 
such that $O^{2'}( \overline{G}^{\sigma} ) = G$. (Contrary to the usage 
in~\cite{GLS}, we write $\overline{G}^{\sigma}$ for the set of $\sigma$-fixed 
points of~$\overline{G}$, rather than $C_{\overline{G}}({\sigma})$.) In 
cases~(i) and~(ii), take $\overline{G} = \PGL_d( \mathbb{F} )$, in 
cases~(iii) and~(iv), take $\overline{G} = E_6( \mathbb{F} )_{\rm ad}$. Then,
with a suitable choice of~$\sigma$, we obtain $\overline{G}^{\sigma} = 
\PGL_d(q), \PGU_d(q), E_6(q)_{\rm ad}$ and ${^2\!E}_6(q)_{\rm ad}$, 
respectively. These groups are not simple under the restrictions on~$q$ 
given above. In fact, $G \unlhd \overline{G}^{\sigma}$ with cyclic factor
of order $\gcd( d, q - 1), \gcd( d, q + 1), 3, 3$, respectively. With this 
setup it is easy to describe the automorphism group of~$G$. Namely,
$$\Aut(G) \cong \overline{G}^{\sigma} \rtimes (\Phi_G \times \Gamma_G),$$
where~$\overline{G}^{\sigma}$ acts on~$G$ by conjugation, and~$\Phi_G$ 
and~$\Gamma_G$ are the groups of field, respectively graph automorphisms of~$G$;
see \cite[Theorem~$2.5.12$]{GLS}. Suppose that $q = 2^f$. In cases~(i) 
and~(iii) the group $\Gamma_G$ has order~$2$ and $\Phi_G$ is cyclic of 
order~$f$; in cases~(ii) and~(iv), the group~$\Gamma_G$ is trivial by 
convention, and $\Phi_G$ is cyclic of order~$2f$.

We will also need to consider the groups dual to~$\overline{G}$.
To describe the irreducible characters of~$G$ it is more convenient to
realize~$G$ as a central quotient of a covering group of~$G$.
If $\overline{G} = \PGL_d( \mathbb{F} )$, respectively 
$E_6( \mathbb{F} )_{\rm ad}$, let $\overline{G}^* = \SL_d( \mathbb{F} )$,
respectively $E_6( \mathbb{F} )_{\rm sc}$. Then there is a Steinberg
morphism~$\sigma$ of~$\overline{G}^*$ dual to~$\sigma$; see 
\cite[Definition~$1.5.17$]{GeMa}. In our cases we have 
${\overline{G}^*}^{\sigma} = \SL_d( q ), \SU_d( q ), E_6(q)_{\rm sc}$ and 
${^2\!E}_6(q)_{\rm sc}$, respectively, and 
$G \cong {\overline{G}^*}^{\sigma}/Z( {\overline{G}^*}^{\sigma} )$.
(Up to finitely many exceptions,~${\overline{G}^*}^{\sigma}$ is the universal
covering group of~$G$.) We may view the irreducible characters of~$G$ as
characters of~${\overline{G}^*}^{\sigma}$ via inflation. Starting with
$\Irr( {\overline{G}^*}^{\sigma} )$ right away does not introduce additional
aspects to be investigated, as a real irreducible character of 
${\overline{G}^*}^{\sigma}$ has $Z( {\overline{G}^*}^{\sigma} )$ in its
kernel.

The set $\Irr( {\overline{G}^*}^{\sigma} )$ is partitioned into Lusztig series
$\mathcal{E}( {\overline{G}^*}^{\sigma}, s )$, where~$s$ runs through the
$\overline{G}^{\sigma}$-conjugacy classes of semisimple elements
of~$\overline{G}^{\sigma}$; see \cite[Definition~$2.6.1$]{GeMa}. The following
lemma, whose proof can be extracted from the literature, is used to parametrize
the real characters of~$G$ of odd degree. Recall that an element of a group is
called \textit{real}, if it is conjugate to its inverse.
\begin{lemma}
\label{lem:6}
Let $\chi \in \Irr( {\overline{G}^*}^{\sigma} )$ be of odd degree and let 
$s \in \overline{G}^{\sigma}$ be semisimple such that 
$\chi \in \mathcal{E}( {\overline{G}^*}^{\sigma}, s )$. 
Then the following statements hold.

\begin{enumerate}[{\rm (a)}, widest=a]
\item We have
$\chi(1) = [\overline{G}^\sigma\colon\!C_{\overline{G}^{\sigma}}(s)]_{2'}$.

\item The characters in $\mathcal{E}( {\overline{G}^*}^{\sigma}, s )$ 
of odd degree correspond, via Lusztig's generalized Jordan decomposition of 
characters, to the irreducible characters of 
$(C_{\overline{G}}( s )/C^{\circ}_{\overline{G}}( s ))^\sigma$, and they all 
have the same degree. (Here, $C^{\circ}_{\overline{G}}( s )$ denotes the 
connected component of $C_{\overline{G}}( s )$.) In particular, if 
$C_{\overline{G}}( s )$ is connected, then~$\chi$ is the unique character in 
$\mathcal{E}( {\overline{G}^*}^{\sigma}, s )$ of odd degree.

\item If~$\chi$ is real, then~$s$ is real in $\overline{G}^{\sigma}$.
If~$s$ is real in $\overline{G}^{\sigma}$ and
$C_{{\overline{G}}}( s )$ is connected, then~$\chi$ is real.
\end{enumerate}
\end{lemma}
Taking $s = 1$, the lemma shows that the trivial character is the unique
character of odd degree in $\mathcal{E}( {\overline{G}^*}^{\sigma}, 1 )$, the set
of unipotent characters of~${\overline{G}^*}^{\sigma}$.
There are analogous compatibility properties as in Lemma~\ref{lem:6}(c) for 
the action of certain automorphisms of~$G$, but these are too technical to state 
here. 
By Lemma~\ref{lem:6}, in order to classify the real irreducible characters of 
odd degree in~${\overline{G}^*}^{\sigma}$, it is necessary to describe the 
conjugacy classes of real elements and their centralizers 
in~$\overline{G}^{\sigma}$. If~$\overline{G}^{\sigma}$ is one of 
$\PGL_d( q )$ or $\PGU_d(q)$, this task is achieved with methods of linear 
algebra. If~$\overline{G}^{\sigma}$ equals $E_6(q)_{\rm ad}$ or 
${^2\!E}_6(q)_{\rm ad}$, we make use of the tables of Frank L{\"u}beck on the 
website~\cite{LL}. 

When the large degree method fails, we will use the restriction method. This is 
based on the following result. Here,~$V$ is an irreducible $\mathbb{R}G$-module 
of odd dimension, $\rho \colon G \rightarrow \GL(V)$ is the representation afforded 
by~$V$, and $n \in \GL(V)$ is an element of finite order normalizing $\rho(G)$.
Moreover,~$\nu$ is the automorphism of~$G$ induced by~$n$. The character 
$\chi = \chi_V$ of~$G$ is viewed as a character of~${\overline{G}^*}^{\sigma}$ 
via inflation.

\begin{lemma}
\label{lem:7}
Suppose that every proper subgroup of~$G$ has the $E1$-property. Let~$\iota$
denote the standard graph automorphism of~$\overline{G}$ of order~$2$. Let 
$s \in \overline{G}^{\sigma}$ be such that 
$\chi \in \mathcal{E}( {\overline{G}^*}^{\sigma}, s )$. Then $(G,V,n)$ has the 
$E1$-property under the following hypotheses.

There is a $\iota$-stable, proper standard Levi 
subgroup~$\overline{L}$ of~$\overline{G}$ and a
$\overline{G}^{\sigma}$-conjugate $s' \in \overline{L}^{\sigma}$ of~$s$, such 
that the following three conditions hold.

\begin{enumerate}[{\rm (i)},widest=iii]
\item The element~$s'$ is real in $\overline{L}^\sigma$.
\item The centralizer~$C_{\overline{L}}( s' )$ is connected.
\item For every $\alpha \in \Aut( \overline{G}^{\sigma} )$
stabilizing~$\overline{L}^{\sigma}$, the following holds: If~$\alpha(s)$ and~$s$
are conjugate in~$\overline{G}^{\sigma}$, then~$\alpha(s')$ and~$s'$ are
conjugate in~$\overline{L}^{\sigma}$.
\end{enumerate}
\end{lemma}
\begin{proof}
(Rough sketch) 
There is a standard parabolic subgroup~$\overline{P}^*$ of $\overline{G}^*$ 
containing the standard Levi subgroup $\overline{L}^*$ as a Levi complement,
such that~$\overline{L}^*$ is dual to~$\overline{L}$. The finite group 
${\overline{L}}^{\sigma}$ is a standard Levi subgroup 
of~$\overline{G}^{\sigma}$, contained in the standard parabolic
subgroup $\overline{P}^{\sigma}$ as a Levi complement. The fact 
that~$\overline{L}$ is $\iota$-stable implies that there is $g \in G$ such that 
${\overline{L}}^{\sigma}$ and $\overline{P}^{\sigma}$ are stable under 
$\alpha := \ad_g \circ \nu$ (where~$\alpha$ is extended to an automorphism
of~$\overline{G}^{\sigma}$ via \cite[Theorem~$2.5.14$(a)]{GLS}). Moreover, with 
a suitable choice of~$g$, the 
automorphism~$\alpha$ of~$G$ lifts to an automorphism~$\alpha^*$ of 
${\overline{G}^*}^{\sigma}$, such that ${\overline{L}^*}^{\sigma}$ and 
${\overline{P}^*}^{\sigma}$ are stable under~$\alpha^*$.

We may assume that $s = s' \in \overline{L}^{\sigma}$. In view of 
Lemma~\ref{lem:6}(b--c), the compatibility of Harish-Chandra 
restriction and the Jordan decomposition of characters implies
that the restriction of~$V$ to~${\overline{P}^*}^{\sigma}$ contains an 
irreducible submodule~$V_1$ of odd dimension, of multiplicity one, and with 
$O_2( {\overline{P}^*}^{\sigma} )$ in its kernel. The analogue of 
Lemma~\ref{lem:6}(c) for the automorphisms~$\alpha$ and~$\alpha^*$ implies
that $V_1$ is stable under~$\alpha^*$, hence under~$\alpha$. The claim follows
from Lemma~\ref{lem:1}, applied to $({\overline{P}^*}^{\sigma}, V_1)$.
\end{proof}

To apply the large degree method based on Proposition~\ref{prop:1}, we need
to estimate the orders $|C_G( \alpha' )|$ for every 
$\alpha' \in \langle \alpha \rangle$ of prime order, for suitable 
$\alpha \in \Aut(G)$. The fixed point groups $C_G( \beta )$ for certain
non-inner automorphisms $\beta \in \Aut(G)$ are described in 
\cite[Propositions~$4.9.1$,~$4.9.2$]{GLS}. This leads to the following results.
\begin{proposition}
\label{prop:5}
Let $V$ be a non-trivial irreducible $\mathbb{R}G$-module of odd dimension
with character~$\chi$. Then the following hold.

\begin{enumerate}[{\rm (a)}, widest=(a)]
\item Suppose that $G = \PSL_d( q )$ or $\PGU_d(q)$ with $d \geq 5$ and 
$q > 4$.
If
$$\chi( 1 ) > q^2\cdot q^{d(d+1)/4},$$
then $(G,V)$ has the $E1$-property.

\item Suppose that $G = E_6(q)$ or ${^2\!E}_6( q )$ with $q > 16$.
If 
$$\chi(1) > q\cdot q^{26},$$
then $(G,V)$ has the $E1$-property.
\end{enumerate}
\end{proposition}
\begin{proof}
(Rough sketch) Recall that $q = 2^f$. Let $\beta \in \Aut(G)$. Then there is 
$g \in G$ such that $\alpha := \ad_g \circ \beta$ has even order and the 
following properties hold.

If $G$ is as in (a), then $|\alpha| \leq 2f(q+1)$ and $|C_G( \alpha' )| 
\leq q^{d(d+1)/2}$ for every $\alpha' \in \langle \alpha \rangle$ of prime 
order. Observe that $q^2 = 2^{2f} > 2f(2^f+1)$ for $f > 2$.

If $G$ is as in (b),  then $|\alpha| \leq 6f$ and $|C_G( \alpha' )| 
\leq q^{52}$ for every $\alpha' \in \langle \alpha \rangle$ of prime 
order. Observe that $q = 2^f > 6f$ for $f > 4$.

The assertions now follow from Proposition~\ref{prop:3}.
\end{proof}
For small values of~$q$ and for $d = 3$ in (a) above, we derive more precise
estimates for the orders of the automorphisms~$\alpha$, to deal with more cases.

For the characters~$\chi$, which do not meet the degree estimates of
Proposition~\ref{prop:5} and their refinements, we usually can apply 
Lemma~\ref{lem:7}. If $G = \PGL_d(q)$,
let $\hat{s} \in \GL_d(q)$ be a real lift of an element 
$s \in \overline{G}^{\sigma}$ parametrizing~$\chi$ as in Lemma~\ref{lem:6}. The 
fact that~$\chi(1)$ is smaller than the estimate required in 
Proposition~\ref{prop:5} implies strong restrictions on~$\hat{s}$, acting on its 
natural vector space. Namely, either the fixed point space of~$\hat{s}$ has 
dimension at least~$d/3$, or~$\hat{s}$ has at most three distinct eigenvalues 
$1, \zeta, \zeta^{-1}$. It is then not difficult to construct the Levi 
subgroup~$\overline{L}$ as required in Lemma~\ref{lem:7}.
A similar approach works for the group $G = \PGU_d( q )$. For the exceptional
groups of type~$E_6$, we use the lists of explicit character degrees computed
by L{\"u}beck and given in~\cite{LL2}, as well as the Chevie system 
\cite{chevie,Michel} for extensive computations in the Weyl group of type $E_6$.
Still, there are numerous small cases, which cannot be handled either way. 
These are treated with computational methods using GAP~\cite{GAP04}. 

\begin{acknowledgement}
The authors thank, in alphabetical order, Thomas Breuer, Xenia Flamm, Meinolf
Geck, Jonas Hetz,
Frank Himstedt,  
Frank L{\"u}beck, Jean Michel, Britta Sp{\"a}th, Andrzej Szczepa{\'n}ski and
Jay Taylor for numerous helpful hints and enlightening conversations.

The first named author also thanks the organizers and sponsors of the IGT 2024
for inviting him to this conference, as well as the organizers and sponsors of
The Twenty Third Andrzej Jankowski Memorial Lecture at the University of 
Gda{\'n}sk, where he presented these results at the accompanying Mini 
conference.
\end{acknowledgement}


\begin{thebibliography}{100}
%
%
%
\bibitem{C2}     {R.~W.~Carter}, {Finite groups of Lie type:
                 Conjugacy classes and complex characters},
                 John Wiley \& Sons, Inc., New York 1985.

%
\bibitem{DDRP}   {\sc K.~Dekimpe, B.~De Rock and P.~Penninckx}, The~$R_\infty$ 
                 property for infra-nilmanifolds, {\em Topol.\ Methods Nonlinear 
		 Anal.} {\bf 34} (2009), 353--373.

%

\bibitem{GAP04}  The GAP Group, GAP -- Groups, Algorithms, and Programming, Version 
                 4.11.1; 2021, \url{http://www.gap-system.org}.

\bibitem{chevie} {\sc M.~Geck, G.~Hiss, F.~L{\"u}beck, G.~Malle, and
                 G.~Pfeiffer}, {\sf CHEVIE}---A system for computing
                 and processing generic character tables,
                 {\em AAECC} {\bf 7} (1996), 175--210.

 \bibitem{GeMa}  {\sc M.~Geck and G.~Malle}, The character theory of finite
                 groups of Lie type, A guided tour, Cambridge Studies in Advanced
                 Mathematics, 187, Cambridge University Press, Cambridge, 2020.

\bibitem{GLS}    {\sc D.\ Gorenstein, R.\ Lyons, and R.\ Solomon},
                 {The classification of finite simple groups, Number 3},
                 Mathematical Surveys and Monographs, AMS, Providence, 1998.

\bibitem{HiLu}   {\sc G.~Hiss and R.~Lutowski}, The eigenvalue one property of 
                 finite groups, Preprint, 2024.

%
%
%
%
%
%
%

\bibitem{LL}     {\sc F.~L{\"u}beck}, {Centralizers and numbers of semisimple
                 classes in exceptional groups of Lie type},
                 \url{http://www.math.rwth-aachen.de/~Frank.Luebeck/chev/CentSSClasses/}.

\bibitem{LL2}    {\sc F.~L{\"u}beck}, {Character Degrees and their 
                 Multiplicities for some Groups of Lie Type of Rank $< 9$},
                 \url{http://www.math.rwth-aachen.de/~Frank.Luebeck/chev/DegMult/index.html}.

%
%
\bibitem{LuSzcz} {\sc R.~Lutowski and A.~Szczepa{\'n}ski}, Holonomy groups of 
                 flat manifolds with the $R_{\infty}$ property, 
                 {\em Fund.\ Math.} {\bf 223} (2013), 195--205.

%
\bibitem{MalleHeight0} {\sc G.~Malle}, Height~$0$ characters of finite groups 
                 of Lie type, {\em Represent.\ Theory} {\bf 11} 192--220, 2007.

\bibitem{MalleExt}{\sc G.~Malle}, Extensions of unipotent characters and the
                 inductive McKay condition, {\em J.~Algebra} {\bf 320} (2008),
                 2963--2980.

%
\bibitem{Michel} {\sc J.~Michel}, The development version of the {\tt CHEVIE}
                 package of {\tt GAP3}, {\em J.~Algebra} {\bf 435} (2015),
                 308--336.


%
%

\bibitem{SzczOut}{\sc A.~Szczepa{\'n}ski}, Outer automorphism groups of 
                 Bieberbach groups, {\em Bull.\ Belg.\ Math.\ Soc.\ Simon 
                 Stevin} {\bf 3} (1996), 585--593. 

%
%
%
\end{thebibliography}
\end{document}